\documentclass[11pt,a4paper]{article}
\usepackage{amsmath} 
\usepackage{amssymb}
\usepackage{amsthm}
\usepackage{dsfont}
\usepackage{a4wide}
\newtheorem{theorem}{Theorem}[section]
\newtheorem{lemma}[theorem]{Lemma}
\newtheorem{proposition}[theorem]{Proposition}

\theoremstyle{definition}
\newtheorem{definition}[theorem]{Definition}
\theoremstyle{remark}
\newtheorem{remark}{Remark}

\numberwithin{equation}{section}

\newcommand{\D}{{\, \rm d}}

\newcommand{\eps}{\varepsilon}

\newcommand{\R}{\mathbb{R}}

\newcommand{\omu}{\nu}
\newcommand{\bJ}{\mathbf{J}}
\newcommand{\bK}{\mathbf{K}}
\newcommand{\bI}{\mathbf{I}}
\newcommand{\bH}{\mathbf{H}}
\newcommand{\bL}{\mathbf{L}}
\newcommand{\bD}{\mathbf{D}}
\newcommand{\bu}{\mathbf{u}}

\newcommand{\be}{\mathbf{e}}
\newcommand{\mH}{\mathcal{H}}

\title{Analysis of a non local model for spontaneous cell polarisation}

\author{Vincent Calvez\thanks{Unit\'e de Math\'ematiques Pures et Appliqu\'ees, CNRS UMR 5669 \& \'equipe-projet INRIA NUMED, \'Ecole Normale Sup\'erieure de Lyon, 46 all\'ee d'Italie, F-69364 Lyon, France.  ({\tt vincent.calvez@umpa.ens-lyon.fr})} \and Rhoda J. Hawkins\thanks{Laboratoire de la mati\`ere condens\'ee, CNRS UMR 7600, Universit\'e Pierre et Marie Curie, 4 Place Jussieu, 75255 Paris Cedex 05 France ({\tt rhoda@lptmc.jussieu.fr})} \and Nicolas Meunier\thanks{MAP5, CNRS UMR 8145, Universit\'{e} Paris Descartes, 45 rue des Saints  P\`{e}res
75006 Paris,
France. ({\tt nicolas.meunier@parisdescartes.fr})} \and Raphael Voituriez.
        \thanks{Laboratoire de la mati\`ere condens\'ee, CNRS UMR 7600, Universit\'e Pierre et Marie Curie, 4 Place Jussieu, 75255 Paris Cedex 05 France ({\tt voiturie@lptmc.jussieu.fr})}}
\begin{document}

\maketitle


\begin{abstract}
In this work, we investigate the dynamics of a non-local model describing spontaneous cell polarisation. It consists in a drift-diffusion equation set in the half-space, with the coupling involving the trace value on the boundary. We characterize the following behaviours in the one-dimensional case: solutions are global if the mass is below the critical mass and they blow-up in finite time above the critical mass. The higher-dimensional case is also discussed. The results are reminiscent of the classical Keller-Segel system in double the dimension. In addition, in the one-dimensional case we prove quantitative convergence results using relative entropy techniques.  This work is complemented with a more realistic model that takes into account dynamical exchange of molecular content at the boundary. In the one-dimensional case we prove that blow-up is prevented. Furthermore, density converges towards a non trivial stationary configuration.
\end{abstract}



\section{Introduction}

Cell polarisation refers generically to a process that enables a cell to switch from a spherically symmetric shape to a state with a prefered axis. Such a phenomenon 
is an essential step for many biological processes and is involved for instance in cell migration, division, or morphogenesis. While the precise biochemical basis of polarisation can vary greatly,  in its early stages polarisation  is always characterised by an inhomogeneous distribution of specific molecular markers. Cell polarisation can be driven by an external asymmetric signal as in the example of chemotaxis, where  a chemical gradient imposes the direction of migration of cells  \cite{alberts}. Another example is given by mating yeast, for which the external signal is a pheromone gradient, which causes the cell to grow an elongation known as a shmoo in the direction of the pheromone source \cite{alberts}.  However observations show that some cellular systems, such as mating yeast, can also polarise spontaneously in absence of external gradients \cite{Wedlich-Soldner2003}. These two distinct polarisation processes, {\it driven} or {\it spontaneous}, are necessary for  cells to fulfil different biological functions. However, so far the conditions under which  a cell can polarise spontaneously or only in response to an external asymmetric forcing   are not well understood.

The molecular basis of  cell polarisation has been much discussed in the biological literature over the past decade, and is likely to involve several processes. It is now widely recognised that the cell cytoskeleton plays a crucial role in cell polarisation. The cell cytoskeleton is a network of long semiflexible filaments made up of protein subunits (mainly actin or microtubules). These filaments act as roads along which motor proteins are able to perform a biased ballistic motion and carry various molecules, in a process which consumes  the chemical energy of adenosine triphosphate ATP.  It is observed that the efficiency of formation of polar caps in yeast, indicating polarisation, is reduced when actin transport is disrupted, and that the polar caps formed are unstable \cite{Wedlich-Soldner2003,Wedlich-Soldner2004,Irazoqui2005}. In the case of neurons, it has been shown that the  polarisation of the growth cone is suppressed when microtubules are depolymerised \cite{Bouzigues2007}. To account for these observations, it is generally argued that the cytoskeleton filaments mediate an effective positive feedback in the dynamics of polarisation markers \cite{Wedlich-Soldner2003}. This arises from the molecular markers not only diffusing in the cell cytoplasm, but also being actively transported by molecular motors along cytoskeleton filaments, the dynamic organisation of which is regulated by the markers themselves.

From the physical point of view, achieving an inhomogeneous  distribution of diffusing molecules without an external asymmetric field as in the case of spontaneous polarisation requires either an interaction between the molecules or a driving force that maintains the system out of equilibrium. In the case of the cell cytoskeleton, it is well known that the hydrolysis of ATP acts as a sustained energy input which drives the system out of equilibrium, and one can therefore hypothesizes on general grounds that spontaneous polarisation in cells stems from non equilibrium processes. Cell polarisation has been the subject of a few theoretical studies  in recent years. Many models rely on reaction-diffusion systems in which polarisation emerges as a type of Turing instability \cite{Iglesias2008,Levine2006,Onsum2007} and some (e.g. \cite{Onsum2007,Wedlich-Soldner2003}) include cytoskeleton proteins as a regulatory factor. However,  the full dynamics of markers  is  generally not considered.

In this article, following the work of \cite{HBPV}, we study a class of models for spontaneous cell polarisation. These models couple the evolution of molecular markers with the dynamics of the cytoskeleton. Namely the markers are assumed to diffuse in the cytoplasm and to be actively transported along the cytoskeleton. The density of molecular markers is denoted by $n(t,x)$. The advection field is denoted by $\bu(t,x)$. This field is obtained through a coupling with the boundary value of $n(t,x)$.

The cell is figured by the half-space $\mH = \R^{N-1}\times (0,+\infty)$. We denote the space variable $x = (y,z)$. The time evolution of the molecular markers follows an advection-diffusion equation:
\begin{equation}\label{eq:2D model}
\partial_t n(t,x) = \Delta n(t,x) - \nabla\cdot\left( n(t,x) {\bf u}(t,x)\right) \, , \quad \, t>0\, , \quad x\in \mH \, .
\end{equation}

\subsection{The one-dimensional case}

We first analyse two different models set on the half-line $(0,+\infty)$. In the simplified version, the advection field is given by $\bu (t,z) = - n(t,0)$. Active transport arises at uniform speed, the speed being given by the value of the density at $z = 0$.

\subsubsection{The simplified model}

The model writes as follows.
\begin{equation}\label{eq1D}
\partial _t n(t,z) = \partial _{zz} n(t,z) +n(t,0) \partial _z n(t,z)\, , \quad t >0\, , \, z\in (0,+\infty)\, ,
\end{equation}
together with the zero-flux boundary condition at $z = 0$:
\begin{equation}\label{cl1D}
\partial _z n (t,0)+n(t,0)^2=0\, . 
\end{equation}
We have formally conservation of molecular content:
\[M =\int_{z>0} n_0(z)\D z = \int_{z>0} n(t,z)\D  z  \, . \]

Solutions of (\ref{eq1D}) may become unbounded in finite time (so-called blow-up). This occurs if the mass $M$ is above the critical mass: $M>1$. In the case $M< 1$, the solution converges to 0. In the critical case $M = 1$ there exists a family of stationary states parametrized by the first moment. The solution converges to the stationary state corresponding to the first moment of the initial condition $\int_{z>0} z n_0(z)\D z$.

\begin{theorem}[Global existence and asymptotic behaviour in the sub-critical and critical cases: $M\leq1$] \label{th:1D} 
Assume that the initial data $n_0$ satisfies both $n_0 \in L^1(( 1 + z)\D z)$ and $\int_{z>0} n_0(z) (\log n_0(z))_+ \D z< + \infty$. Assume in addition that $M\leq 1$, then there exists a global weak solution (in the sense of Definition \ref{def:weak}) that satisfies the following estimates for all $T>0$,
\begin{eqnarray*}
\sup_{t\in (0,T)} \int_{z>0} n(t,z) (\log n(t,z))_+ \D z &<& +\infty\, ,\\
\int_0^T\int_{z>0}n(t,z) \left( \partial_z \log n(t,z) \right)^2 \D  z \D t &<& +\infty\, . 
\end{eqnarray*}
In the sub-critical case $M<1$ the solution strongly converges in $L^1$ towards the self-similar profile $G$ given by (\ref{eq:stat state rescaled}) in the following sense: 
\[ \lim_{t\to +\infty }\left\|n(t,z) - \frac{1}{\sqrt{1+ 2t}} G\left(\frac{ z}{\sqrt{1+ 2t}}\right) \right\|_{L^1} = 0\, .  \]
In the critical case $M = 1$, assuming in addition that the second moment is finite $\int_{z>0} z^2n_0(z)\D z<+\infty$, the solution strongly converges in $L^1$ towards a stationary state $\alpha \exp(-\alpha z)$, where $\alpha^{-1} = \int_{z>0} z n_0(z)\D  z$.\\
\end{theorem}

\begin{theorem}[Blow-up of weak solutions: $M>1$] \label{th:1D BU} Assume $M>1$. Any weak solution with non-increasing initial data $n_0$ blows-up in finite time.\\
\end{theorem}

In the present biological context, blow-up of solutions is interpreted as polarisation of the cell. Indeed there is a strong instability driving the system towards an inhomogeneous state. 

In Section \ref{secvariants}, we present analogous blow-up results in the case of a finite interval $z \in (0,L)$ or finite range of action. 

\begin{remark}
Such a critical mass phenomenon (global existence {\em versus} blow-up) has been widely studied for the Keller-Segel system (also known as the Smoluchowski-Poisson system) in two dimensions of space \cite{BDP,P}. The equation (\ref{eq1D}) represents in some sense a caricatural version of the classical Keller-Segel system in the half-line $(0,+\infty)$. Note that there exist other ways to mimick the two dimensional case in one dimension \cite{CPS,CieslakLaurencot}. 
\end{remark}

\begin{remark}
There is a strong connection between the equation under interest here (\ref{eq1D}) and the one-dimensional Stefan problem. The later writes indeed \cite{HV1}:
\[\left\{\begin{array}{l}
\partial_t u(t,z) = \partial_{zz} u(t,z) \, , \quad \, t>0\, , \, z\in (-\infty,s(t))\, , \\
\lim_{z\to -\infty}\partial_zu (t,z) = 0 \, , \quad u(t,s(t)) = 0\, , \quad \partial_z u (t,s(t)) = -s'(t)\, .
\end{array}\right.\]
The temperature is initially non-negative: $u(0,z) = u_0(z)\geq 0$. By performing the following change of variables: $\phi(t,z) = - u(t,s(t)-z)$, we get an equation that is linked to (\ref{eq:1D}) by $n(t,z) = \partial_z \phi(t,z)$. This connection provides some insights concerning the possible continuation of solutions after blow-up \cite{HV1}. This question has raised a lot of interest in the past recent years \cite{HV2,V1,V2,DS}. It is postulated in \cite{HV1} that the one-dimensional Stefan problem is generically non continuable after the blow-up time. 
\end{remark}

\subsubsection{The model with dynamical exchange of markers at the boundary}

The boundary condition (\ref{cl1D}) turns out to be unrealistic from a biophysical viewpoint. This claim is emphasized by the possible occurence of blow-up in finite time. On the way towards a more realistic model, we distinguish between cytoplasmic content $n(t,z)$ and the concentration of trapped molecule on the boundary at $z=0$: $\mu(t)$. Then the exchange of molecules at the boundary is described by very simple kinetics: 
\[\frac {\D}{\D t} \mu(t)= n(t,0)- \gamma \mu(t)\, .\]
The transport speed is modified accordingly: $\bu(t,z) = - \mu(t)$. 
The model writes:
\[\left\{\begin{array}{l} 
\partial _t n (t,z)=\partial _{zz} n (t,z) +\mu (t) \partial _z n (t,z) \, , \quad t >0\, , \, z\in (0,+\infty) \medskip\\ 
\partial _z n (t,0)+\mu(t) n(t,0) = \frac d{dt} \mu(t) \, .
\end{array}\right. \]

The flux condition on the boundary ensures the conservation of molecular content. Denoting $m(t) = \int_{z>0} n(t,z)\D  z$ the partial mass of cytoplasmic markers, we have:
\[M = \mu_0 + m_0 = \mu(t) +  m(t) \, .\]

Since the transport speed is bounded, $\mu(t)\leq M$, we clearly have global existence of solutions for any mass $M>0$. We can precise the asymptotic behaviour in the super-critical case $M>1$. This is the purpose of the following Theorem.\\

\begin{theorem}\label{th:long time critical1} Assume that the initial data $n_0$ satisfies both $n_0 \in L^1(( 1 + z)\D z)$ and $\int_{z>0} n_0(z) (\log n_0(z))_+ \D z< + \infty$. Assume the mass is super-critical $M> 1$. The partial mass $m(t)$ converges to 1 and the density $n(t,z)$ strongly converges in $L^1$ towards the exponential profile $(M-1) e^{-(M-1)z}$.
\end{theorem} 

\subsection{The higher-dimensional case}

In the higher dimensional case $N\geq 2$ we only partially analyse simplified models such as \eqref{eq1D} where the transport speed is directly computed from the trace value $n(t,y,0)$. Equation (\ref{eq:2D model}) is complemented with the zero-flux boundary condition:
\begin{equation}\label{cl:ajout:nico}
\partial_z n (t,y,0) - n(t,y,0)\bu (t,y,0)\cdot  {\bf e}_z =0\,, \quad y \in \R^{N-1}\, . 
\end{equation}
We have formally conservation of the molecular content:
\[M  =\int_{\mH} n_0(x)\D  x = \int_{\mH} n(t,x)\D  x\, . \]

Following \cite{HBPV} we make the distinction between two possible choices for the advection speed ${\bf u}$.
In the {\bf transversal case}, the field $\bu$ is normal to the boundary:
\begin{equation}\label{eq:u1}
\bu(t,y,z)  = - n(t,y,0)  {\bf e}_z\, . 
\end{equation}
This corresponds to a particular orientation of the cytoskeleton, modelling the microtubules. Indeed microtubules are very rigid filaments whose bending length is larger than the typical size of yeast cells. 

In the {\bf potential case}, the field $\bu$ derives from a harmonic potential. The source term of the potential is located on the boundary:
\begin{equation}\label{eq:u2}
\bu(t,x)  = \nabla c(t,x) \, , \quad \mbox{where}\quad \left\{\begin{array}{rl} -\Delta c(t,x) &= 0\, ,\medskip \\ - \partial_z c(t,y,0) &= n(t,y,0)\, . \end{array}\right.
\end{equation}
This corresponds to another orientation of the cytoskeleton, modelling the actin network. Indeed the actin networks is a diffusive network where orientations are mixed up.  
In dimension $N=1$, observe that the two choices (\ref{eq:u1}) and (\ref{eq:u2}) coincide. 

In dimension $N\geq 2$, we state global existence for small initial data. The criteria are identical for the two possible choices of the advection field (\ref{eq:u1}) or (\ref{eq:u2}). This is a consequence of the two common features: both fields are divergence free and possess the same normal component at the boundary.\\ 

\begin{theorem}[Global existence in dimension $N\geq 2$]\label{thdim2}
Assume that the advection field satisfies the two following conditions: $\nabla\cdot \bu \geq 0$ and $\bu(t,y,0)\cdot \be_z = n(t,y,0)$. Assume that the initial data $n_0$ satisfies both $n_0 \in L^1(( 1 + |x|^2)\D x)$ and $\|n_0\|_{L^N}$ is smaller than some constant $c_N$ depending only on the dimension $N$. Then there exists a global weak solution to (\ref{eq:2D model}) and (\ref{cl:ajout:nico}). \\ 
\end{theorem}

Notice that  both conditions $\nabla\cdot \bu \geq 0$ and $\bu(t,y,0)\cdot \be_z = n(t,y,0)$ are fulfilled in (\ref{eq:u1}) and (\ref{eq:u2}).\\

\begin{theorem}[Blow-up in dimension $N\geq 2$]\label{th2dim2}
Assume that $n(t,x)$ is a strong solution to (\ref{eq:2D model}) which verifies:
\begin{itemize}
\item $\partial_z n(t,x) \leq 0$ for all $x\in \mH$ and $t>0$ when the advective field is given by (\ref{eq:u1}),
\item $\partial_z n(t,x) \leq 0$ and for all $x\in \mH$ and $t>0$, the matrix $A(t,x) = x \, \otimes \,  \partial_z \nabla_y \log n(t,x)$ satisfies $A^T + A \geq 0$ (in the matrix sense)  when the advective field is given by (\ref{eq:u2}).
\end{itemize} Assume in addition that the second momentum is initially small enough: there exists a constant  $C_N$ depending only on the dimension such that $\int_{x\in \mH } |x|^2 n_0(x)\D  x \leq C_N M^{\frac{N+1}{N-1}} $. Then the maximal time of existence of the solution is finite.\\
\end{theorem}

\paragraph{Open questions} We end this introductory Section with some open questions that we are not able to resolve.
(i) 
Obtain a rate for the convergence in relative entropy in Theorem \ref{th:1D} for the cases $M = 1$ and $M<1$.
(ii) Prove blow-up for the systems (\ref{eq:2D model})--(\ref{eq:u2}) with large initial data  without any monotonicity assumption on the density $n(t,x)$.

The outline of the paper is as follows. In Section~\ref{secdim1}, we analyse with full details the one-dimensional case. In section~\ref{secvariants} we study some variants of blow-up criteria in the one-dimensional case. In Section~\ref{sec:ODE/PDE}, we study a model with flux of markers at the boundary in the one-dimensional case. In Section~\ref{secdimsup}, we analyse the higher dimensional case.

Results in the one-dimensional case have been announced in the note \cite{CalvezMeunier}.

\section{The boundary Keller-Segel (BKS) equation in dimension $N = 1$}\label{secdim1}

In this Section we study the following equation,
\begin{equation}\label{eq:1D}
\left\{\begin{array}{l}
\partial _t n(t,z) = \partial _{zz} n(t,z) +n(t,0) \partial _z n(t,z)\, , \quad t >0\, , \, z\in (0,+\infty)\, ,\medskip\\
\partial _z n (t,0)+n(t,0)^2=0\, , 
\end{array}\right.
\end{equation}
and we prove Theorems \ref{th:1D} and \ref{th:1D BU}. More precisely, in Sections \ref{sec:M<1} we prove the existence of a global weak solution for $M\le 1$. Then in Section \ref{sec:BU} we prove the blow up character in the case $M>1$. 

We begin with a proper definition of weak solutions, adapted to our context.
\begin{definition}\label{def:weak}
We say that $n(t,z)$ is a weak solution of (\ref{eq:1D})  on $(0,T)$ if it satsifies:
\begin{equation}
n\in L^\infty(0,T;L^1_+(\R_+))\, , \quad \partial_z n \in L^1((0,T)\times \R_+)  \, , \label{eq:flux L1}
\end{equation}
and $n(t,z)$ is a solution of (\ref{eq:1D}) in the sense of distributions in $\mathcal D'(\R_+)$
\end{definition}

Since the flux $(\partial_z n(t,z) + n(t,0) n(t,z))$ belongs to $ L^1((0,T)\times \R_+)$, the solution is well-defined in the distributional sense under assumption (\ref{eq:flux L1}). In fact we can write
$\int_0^T n(t,0) \D t   = - \int_0^T\int_{z>0} \partial_z n(t,z)\D z\D t$.

Weak solutions in the sense of Definition \ref{def:weak} are mass-preserving: 
\[ M =\int_{z>0} n_0(z)\D  z = \int_{z>0} n(t,z)\D  z\, .\] 
The proof closely follows the arguments of the next Lemma which is concerned with moment growth.\\

\begin{lemma}[Moment growth]\label{Moment growth}
Assume $n(t,z)$ is a weak solution of (\ref{eq:1D}). Assume in addition that $z n_0\in L^1(\R_+)$. Then the following identity holds true: 
\begin{equation}\label{faible11}
 \int _{z>0} z n(T,z) \D z = \int _{z>0} z n_0(z) \D z + \int _0^T \left(1-\int_{z>0} n(t,z) \D z\right)n(t,0) \D t\, .
\end{equation}
\end{lemma}

\begin{proof}
Consider the  approximation function $\chi(z)$ which verifies $\chi(z)=1$ if $0\le z\le 1$, $\chi(z)=0$ if $z\ge 2$, which is smooth and non-negative everywhere.  
Define the family of functions $(\varphi_\varepsilon )_\varepsilon$ by $\varphi_\varepsilon(z) = z \chi(\eps z )$. 
We recall the weak formulation:
\begin{eqnarray*}
\int _{z>0} n(T,z)\varphi_\varepsilon(z)\D  z &=& \int _{z>0} n_0(z)\varphi_\varepsilon(z)\D z \\
& & - \int_0^T \int _{z>0} \left( \partial _{z} n(t,z) +n(t,0)  n(t,z)\right) \varphi_\varepsilon'(z)\D  z \D t\, . 
\end{eqnarray*}
The function $\varphi_\eps(z)$ converges monotically to $z$ as $\eps\to 0$, hence from the 
monotone convergence theorem, we deduce that $z n(T,z)\in L^1$. 

The function $\varphi_\eps'(z) = \chi(\eps z) + \eps z \chi'(\eps z)$ is bounded in $L^\infty$ uniformly in $\eps$ and it converges to 1 $a.e$. Since $n(\cdot, 0) n  \in L^1((0,T)\times R _+)$ and $\partial_z n \in L^1((0,T)\times \R_+)$, from Lebesgue's dominated convergence theorem, it follows that 
\begin{eqnarray*}
\lim _{\varepsilon \to 0}  \int_0 ^T \int _{z>0} \varphi'_{\varepsilon} (z) n(t,0)n(t,z)\D  z \D t & =& \int _0^T \int_{z>0} n(t,0) n(t,z) \D  z\D t \,, \\
\lim _{\varepsilon \to 0} \int_0 ^T \int _{z>0}  \varphi '_\varepsilon  \left(z\right) \partial _z n(t,z) \D  z \D t  & =& \int_0 ^T \int _{z>0}    \partial _z n(t,z) \D  z\D t  = - \int_0 ^T  n(t,0) \D t\, . 
\end{eqnarray*}
\end{proof}

\subsection{Global existence for sub-critical mass $M< 1$}\label{sec:M<1}

\subsubsection{A priori estimates}

Our next result is concerned with the derivation of a priori bounds for solutions to (\ref{eq:1D}) in the classical sense.\\

\begin{proposition}[Main a priori estimate]\label{apriori}
Let $n$ be a classical solution to (\ref{eq:1D}). If  $M<1$, then the following estimate holds true for some $\delta>0$ and for all $t\in (0,T)$:
\begin{eqnarray}
  \int _{z>0}n(t,z) (\log n (t,z))_+   \D z +  \delta \int _{0}^t\int _{z>0} n(s,z)(\partial_z \log n (s,z))^2   \D z\D  s \nonumber \\
 \leq   \int_{z>0} n_0(z)\left(\log n_0(z)\right) _+ \D  z +  \int_{z>0} z n_0(z) \D z +  C(T)\, . \label{eq:main estimate} 
\end{eqnarray}
\end{proposition}

\begin{proof}
We first derive the following trace-type inequality.
\begin{eqnarray}
 n(t,0)^2 &=& \left( \int_{z>0} \partial_z n(t,z) \D z \right)^2 \nonumber \\
\label{ineq:trace:0}
 & \le &\left(\int_{z>0}n(t,z) \D z \right) \left(\int_{z>0}n(t,z) \left( \partial_z \log n(t,z) \right)^2  \D z \right) \, .
\end{eqnarray}
We compute the evolution of the entropy   
\begin{eqnarray}
  \frac{ \D }{ \D t} \int_{z>0} n(t,z)\log n(t,z) \D  z  & = &\int _{z>0} \partial _t n(t,z) \log n(t,z)  \D  z \nonumber \\
 & = &- \int _{z>0} \left(\partial_z n (t,z) + n(t,0) n(t,z)\right) \frac{\partial_z n(t,z) }{n(t,z)}  \D  z \nonumber 
 \\& 
 =& -  \int _{z>0} n(t,z)\left( \partial_z \log n(t,z) \right)^2 \D  z   +    n(t,0) ^2 \, . 
\label{eq:entropy dissipation}
\end{eqnarray}
The two contributions are competing. We estimate the balance using inequality (\ref{ineq:trace:0}). 
\begin{equation*} 
  \frac{ \D }{ \D t} \int_{z>0} n(t,z)\log n(t,z) \D  z    \leq  (M-1) \int _{z>0} n(t,z)\left( \partial_z \log n(t,z) \right)^2 \D  z \, .
\end{equation*}
On the contrary to the classical two-dimensional Keller-Segel equation, the dissipation of entropy gives directly the sharp criterion on the mass. There is no need to seek a free energy as in \cite{BDP} (and references therein). 
To control the negative part of the entropy, we use the following Lemma adapted from \cite{BDP, Calvez.Corrias.Ebde}.\\

\begin{lemma}\label{BDP}
For any $f\in L^1_+(\R _+, (1+z)\D z)$, if $\int  f\log f <+\infty $, then $ f\log f $ is in $ L^1(\R _+)$ and for all $\alpha >0$, the following inequality holds true:
\begin{equation} \label{eq:carleman} 
\int_{z>0}  f(z)(\log f(z))_+ \D z \leq \int_{z>0}  f(z) \left( \log f(z)+  \alpha z \right) \D z+\frac{1}{\alpha e}  
\, .
\end{equation}
\end{lemma}
\begin{proof}
Let $\overline {f} =f \mathds{1}_{f\le 1}$ and $m =\int_{z>0} \overline{f}(z)\D z$. We build up the relative entropy between $\overline f$ and $\alpha e^{-\alpha z}$.
\[ \int_{z>0} \overline {f}(z)\left(\log \overline {f}(z) +\alpha z\right) \D z =  \int_{z>0}  \frac{\overline {f}(z)}{\alpha e^{-\alpha z}} \log \left(\frac{\overline {f}(z)}{\alpha e^{-\alpha z}}\right)\alpha e^{-\alpha z} \D z+m \log \alpha \, .\] 
Using Jensen's inequality, we deduce that 
\begin{eqnarray*}
& &\int_{z>0}  \frac{\overline {f}(z)}{\alpha e^{-\alpha z}} \log \left(\frac{\overline {f}(z)}{\alpha e^{-\alpha z}}\right)\alpha e^{-\alpha z} \D z\\
& &\qquad \qquad  \qquad \ge  \left(\int_{z>0}  \frac{\overline {f}(z)}{\alpha e^{-\alpha z}}\alpha e^{-\alpha z} \D z \right)\log \left(\int_{z>0} \frac{\overline {f}(z)}{\alpha e^{-\alpha z}}\alpha e^{-\alpha z} \D z\right)\\
& & \qquad \qquad  \qquad  = m \log m \, .
\end{eqnarray*} 
Therefore, 
\[ \int_{z>0} \overline {f}(z)\log \overline {f}(z)\D z +\alpha \int_{z>0}  z \overline {f}(z) \D z   \ge   m \log \left( \alpha m \right)   \ge   -\frac{1}{\alpha e }  \, .\]
Using 
$$ \int_{z>0}f(z)(\log f(z))_+ \D z =  \int_{z>0}  f(z) \log f(z) \D z-  \int_{z>0}  \overline {f}(z) \log \overline {f}(z) \D z \, ,$$
this completes the proof of Lemma \ref{BDP}.
\end{proof}

Let us now estimate the first moment. Recalling (\ref{faible11}), we deduce that
\begin{eqnarray}
\int_{z>0} z n(t,z) \D z &\le &\int_{z>0} z n_0(z) \D z+ \int _0^t n(s,0) \D s \nonumber \\
&\le & 
\int_{z>0} z n_0(z) \D z+\frac{T}{4\delta'} + \delta' \int _0^t n(s,0)^2 \D s \, , \nonumber \\
&\le & \int_{z>0} z n_0(z) \D z+\frac{T}{4\delta'} \nonumber \\
& &\qquad   \qquad+ \delta'  \int _0^t \int_{z>0}n(s,z) \left( \partial_z \log n(s,z) \right)^2 \D  z \D s \, . 
\label{ineq:hyper}
\end{eqnarray}
Combining (\ref{eq:entropy dissipation}), (\ref{eq:carleman}) and (\ref{ineq:hyper}) with $\alpha = 1$ we obtain that
\begin{eqnarray*}
 \int_{z>0} n(t,z)\left(\log n(t,z)\right) _+ \D  z  + (1- M -   \delta')\int_0^t\int_{z>0}n(s,z) \left( \partial_z \log n(s,z) \right)^2 \D  z \D s \\ 
 \leq  \int_{z>0} n_0(z)\log n_0(z) \D  z +   \int_{z>0} z n_0(z) \D z + \frac{1}{  e} + \frac{  T}{4\delta'} \, .
\end{eqnarray*}
Since $M<1$ we can choose $\delta'>0$ such that (\ref{eq:main estimate}) holds. 
\end{proof}

\subsubsection{Regularization procedure}

To prove existence of weak solutions in the sense of Definition \ref{def:weak} we perform a classical regularization procedure. We carefully choose our function spaces in order to end up with minimal assumptions on the initial data.
We introduce
\[ a^\eps(t) = \int  _{z>0} \phi_\eps(z) n^\eps(t,z) \D  z\,, \]
where  $\phi_\eps$ is an approximation to the identity. We have formally $a^\eps(t) \to n (t,0)$ as $\eps\to 0$. 

We consider the following regularized problem
\begin{equation}\label{eq:1D:reg}
\left\{\begin{array}{l}\partial _t n^\eps(t,z) = \partial _{zz} n^\eps(t,z) +a^\eps(t) \partial _z n^\eps(t,z)\, , \medskip\\
\partial _{z} n^\eps(t,0) +a^\eps(t)  n^\eps(t,0)=0\,.
\end{array}\right.
\end{equation}
Our aim is to extend the main {\em a priori} estimate (\ref{eq:main estimate}) to the regularized problem (\ref{eq:1D:reg}). We  check that 
\[ a^\eps(t) = - \int _{z>0} \phi_\eps(z) \int _{y=z}^{+\infty}\partial_z n^\eps(t,y)\D  y \D  z  \leq  \int _{z>0} |\partial_z n^\eps(t,z)|\D  z \, .\]
Thus the following inequality replaces (\ref{ineq:trace:0}):
\[a^\eps(t) n^\eps(t,0)   \le M \left(\int_{z>0}n^\eps(t,z) \left( \partial_z \log n^\eps(t,z) \right)^2  \D z \right) \, .\]
On the other hand the moment growth estimate only relies on the diffusion contribution. We have accordingly,
\begin{eqnarray} \label{eq:moment eps}
\int_{z>0} z n^\eps(t,z) \D z &\leq &  \int_{z>0} z n_0(z) \D z+\frac{T}{4\delta '} \\
& & + \delta ' \int _0^t \int_{z>0}n^\eps(s,z) \left( \partial_z \log n^\eps(s,z) \right)^2 \D  z \D s \, .\nonumber 
\end{eqnarray}
It is then straightforward to justify (\ref{eq:main estimate}) for the regularized solution $n^\eps$ in the line of Proposition \ref{apriori}. There exists $\delta >0$ such that
\begin{eqnarray}
  \int _{z>0}n^\eps(t,z) (\log n^\eps (t,z))_+ \,  \D z +  \delta \int _{0}^t\int _{z>0} n^\eps(s,z)(\partial_z \log n^\eps (s,z))^2  \,  \D s\D  z  \nonumber \\
 \leq   \int_{z>0} n_0(z)\left(\log n_0(z)\right) _+ \D  z + \int_{z>0} z n_0(z) \D z +  C(T)\, . \label{eq:main estimate eps} 
\end{eqnarray}

\subsubsection{Time compactness}
\label{sec:aubin}

Passing to the limit as $\eps\to 0$, the main difficulty lies in the nonlinear term $a^\varepsilon(t)\partial _z n^\varepsilon(t,z)$. We need some compactness to proceed further. It is provided by the Aubin-Simon Lemma, see \cite{Aubin1963,Lions1969,Simon}.\\

\begin{lemma}[Aubin-Simon]\label{AubinLions}
Let $X\subset B \subset Y$ be Banach spaces such that the embedding $X\subset B$ is compact. Assume that the set of functions $\mathcal F$ satisfies: $\mathcal F$ is bounded in $L^2(0,T;X)$ and $\partial_t f$ is uniformly bounded in $L^2(0,T;Y)$. Then $\mathcal F$ is relatively compact in  $L^2(0,T;B)$.\\
\end{lemma}

The natural choice for spaces in our context would be $\widetilde X =  W^{1,1}(\R _+) $ and $B =  \mathcal{C}^{0}(\R _+)$ (up to the decay problem at infinity). However, due to the possible apparition of jumps, the embedding $\tilde X \subset B$ is not compact. Using the entropy estimate (\ref{eq:main estimate eps}) we are able to modify the space $\widetilde X$ in order to make the embedding $X\subset B$ compact. The crucial point is to obtain an equi-continuity condition weaker than any H\"older condition, in the spirit of \cite[Theorem 8.36]{Adams}. 

\begin{lemma} \label{lem:I(N)}
Assume $\mathcal F$ is a set of non-negative bounded functions in the following sense: there exists a constant $A >0$ such that for all $f\in \mathcal F$ 
\[   \sup_{t\in (0,T)} \int_{z>0} f(t,z) (\log f(t,z))_+\D z \leq A \, , \quad \int_0^T \int_{z>0}f(t,z) \left( \partial_z \log f(t,z) \right)^2 \D z \D t \leq A\, . \] Then there exists a continuous function $\eta:\R_+\to \R_+$ and a constant $A'$ depending on $A$ such that $\eta(0) = 0$ and for all function $f\in \mathcal F$ we have
\begin{equation} 
\int_0^T \left( \sup_{x\neq y} \frac{|f (t,y) - f(t,x)|}{\eta(y - x)}\right)^2\D t \leq A'\, . \label{eq:equicont}
\end{equation}
\end{lemma}

\begin{proof}
First for $x<y$ we have that
\begin{eqnarray} 
|f(t,y) - f(t,x)|^2 & \leq & \left(\int_x^y |\partial_z f (t,z)|\D z\right)^2 \nonumber \\
&\leq & \left( \int_x^y f(t,z)\D z\right) \left(\int_{z>0}f(t,z) \left( \partial_z \log f(t,z) \right)^2  \D z \right) \,. \label{eq:embedding}
\end{eqnarray}

We use the Jensen's inequality for $x<y$:
\begin{eqnarray*}
\left( \frac{1}{y-x}\int_x^y f(t,z)\D z \right) \log\left(\frac{1}{y-x}\int_x^y f(t,z)\D z \right)_+ 
&\leq & \frac{1}{y-x}\int_x^y f(t,z) (\log f(t,z))_+\D z \\ 
& \leq &\frac{ A}{y - x}\, . 
\end{eqnarray*}
We can invert this inequality to get:
\begin{equation} 
\frac{1}{y-x}\int_x^y f(t,z)\D z \leq \Phi\left ( \frac{ A}{y - x} \right)\, ,  \label{eq:subholder}
\end{equation}
where $\Phi:[0,+\infty) \to [1,+\infty)$ is the reciprocal bijection of $x (\log x)_+$. 
We define $\eta(z) = z \Phi(z^{-1}A)$. Clearly $\eta(z)\to 0$ as $z\to 0$ since $\Phi$ is sublinear. Combining (\ref{eq:embedding}) and (\ref{eq:subholder}), we deduce the estimate (\ref{eq:equicont}).
\end{proof}

We denote $\mathcal C^{0,\eta}$ the space of functions having modulus of continuity controlled by $\eta$:
\[ \mathcal C^{0,\eta} = \left\{ g \in \mathcal C^0\, : \,  \sup_{x\neq y} \frac{|g (y) - g(x)|}{\eta(y - x)} < +\infty  \right\}\, . \]
The injection $\mathcal C^{0,\eta} \subset \mathcal C^0$ is compact on bounded intervals \cite{Adams}. 
The behaviour of functions outside bounded intervals in our context is controlled by the following estimate which is a consequence of (\ref{eq:embedding}) as $y\to +\infty$:
\begin{equation}
  |f(t,x)|^2  \leq \frac{ 1}{x} \left( \int_{z>0} z f(t,z)\D z\right) \left(\int_{z>0}f(t,z) \left( \partial_z \log f(t,z) \right)^2  \D z \right) \, . \label{eq:bound infty}
\end{equation} 

The last requirement in the Aubin-Simon Lemma consists in getting very weak estimate for the time derivative $\partial_t n^\eps$. We can write
\[ \partial_t  n^\eps(t,z) + \partial_z j^\eps(t,z) = 0\, ,\] where $j^\eps(t,z) = \partial_z n^\eps(t,z) + a^\eps(t)n^\eps(t,z) $ is uniformly bounded in $L^2(0,T;L^1(\R_+))$, due to (\ref{eq:main estimate eps}) and the following inequalities: 
\[ \|a^\eps\|^2_{L^2(0,T)}\leq \|\partial_z n^\eps\|^2_{L^2\left(0,T;L^1(\R_+)\right)} \leq M \int_0^T \int_{z>0} n^\eps(t,z)(\partial_z \log n^\eps (t,z))^2  \,  \D t\D  z\, . \]
Hence $\partial_t n^\eps$ is uniformly bounded in $L^2\left(0,T;(W^{1,\infty}(\R_+))^\prime\right)$.

We introduce some useful functional spaces, endowed with their corresponding norms:
\begin{eqnarray*}
 X &=&  \{ g \in \mathcal{C}^{0,\eta}(\R _+) : \, z^{1/2}g(z) \in L^\infty(\R _+) \}\, ,\\
 B &=&  \mathcal{C}^{0}(\R _+)\cap L^\infty(\R _+) \, ,\\
 Y &=& \left(W^{1, \infty }(\R _+)\right) ^\prime\, .
\end{eqnarray*} 
It is straightforward to check that $X$ is compactly embedded in $B$. 

Combining the above estimates (\ref{eq:moment eps}--\ref{eq:main estimate eps}--\ref{eq:equicont}--\ref{eq:bound infty}) we obtain that $n^\eps$ is bounded in $L^2(0, T;X)$ uniformly with respect to $\eps$. 
The Aubin-Simon Lemma ensures that, up to extracting a subsequence, $n^\eps$ converges strongly in $L^2(0, T;B)$ towards some $n$. 
From uniform convergence of $n^\eps$, we deduce that $a^\eps(t) \to n(t,0)$ strongly in $L^2(0,T)$. Hence we can pass to the limit in the nonlinear term $a^\eps(t) n^\eps(t,z)$ in the weak formulation. 

To conclude we  verify that the {\em a priori} estimates given in Proposition \ref{apriori} are valid after passing to the limit $\eps\to 0$. From the strong convergence in $L^2(0,T;B)$ we deduce that, up to extracting a subsequence that we do not relabel,   
\[\lim_{\eps \to 0} \int_{z>0}  n^\eps(t,z) \left( \log n^\eps(t,z)\right) _+ \D z  =  \int_{z>0}  n (t,z) \left( \log n (t,z)\right) _+ \D z   \, , \quad {\rm a.e.}\; t\in (0,T)\, . \]
On the other hand, 
we use the convex character of the functional (see \cite{BDP} and the references therein) $\int_{z>0} f(z) \left( \partial _z \log f(z) \right)^2 \D z = 4 \int_{z>0}  \left( \partial _z \sqrt{ f(z)} \right)^2 \D z $. 
We have finally,
\[\liminf_{\eps \to 0} \int_0^t \int_{z>0}  n^\eps(s,z) \left( \partial _z \log n^\eps(s,z) \right)^2 \D z \D s \geq 
\int_0^t \int_{z>0}  n(s,z) \left( \partial _z \log n(s,z) \right)^2 \D z \D s  \, .\]
So the {\em a priori} estimate (\ref{eq:main estimate}) is valid a.e. $t\in (0,T)$.

\subsection{Long-time behaviour for the critical and the subcritical cases}\label{sec:StSt}

In this Section, we investigate long-time behaviour of solutions in the case $M\leq1$ using entropy methods. We distinguish between the critical case (no need to rescale) and the sub-critical case (self-similar diffusive scaling).

We stress out that the method for proving global existence in the critical case $M=1$ strongly relies on the entropy estimate \eqref{eq:dissip ent M=1}. This is why we analyse the global existence and the long time behaviour all in all.

\subsubsection{The critical case: global existence and asymptotic convergence}
\label{sec:M=1}

The main inequality we have used so far in order to prove global existence is (\ref{ineq:trace:0}). Equality occurs if $\log n(t,z)$ is linear w.r.t. $z$: there exists $\alpha (t)>0$ such that $n(t,z) = M \alpha(t) \exp(-\alpha(t) z)$. In fact the boundary condition (\ref{eq:1D}) implies $M = 1$. On the other hand the stationary states to equation (\ref{eq:1D}) are precisely the one-parameter family:
\[h_\alpha(z)=\alpha \exp\left(-\alpha z\right)\, , \quad \alpha>0 \, .\]
This motivates to introduce the relative entropy: 
\[
\bH(t) =\int_{z>0}\frac{n(t,z)}{h_\alpha (z)} \log \left(\frac{n(t,z)}{h_\alpha (z)}\right)h_\alpha (z)\D  z \\
=\int_{z>0} n(t,z)  \log n(t,z) \D z + \alpha \bJ(t) - \log \alpha \, .\]

Recalling (\ref{faible11}), we notice that the first momentum of density is conserved in the case $M = 1$: $\bJ(t) = \bJ(0)$.
This prescribes the value for $\alpha$ provided we can pass to the limit $t\to \infty$: $\alpha^{-1} =  \bJ(0)$.
We also recall the formal computation giving the time evolution of the relative entropy (\ref{eq:entropy dissipation}):
\begin{eqnarray}
\frac{\D }{\D t} \bH(t)&=& - \int_{z>0} n(t,z)\left( \partial_z \log n(t,z) \right)^2 \D  z   +    n(t,0) ^2 \label{eq:dissip ent M=1} \\
&=& - \int_{z>0} n(t,z)\left(\partial_z \log n(t,z) + n(t,0)\right)^2 \D  z  \leq   0\, . \nonumber
\end{eqnarray}
The Jensen's inequality yields $\bH(t)\geq 0$, so we have $0\leq \bH(t)\leq \bH(0)$.
We deduce from Lemma \ref{BDP} that the quantity $\int_{z>0} n(t,z)  (\log n(t,z))_+ \D z$ is uniformly bounded by some constant denoted by $C_0$:
\[\int_{z>0} n(t,z)  (\log n(t,z))_+ \D z \leq C_0 \, , \quad {\rm a.e.}\; t\in (0,+\infty)\, .\]

The method of proving convergence in relative entropy towards $h_\alpha$ is as follows. We first gain {\em a priori} estimates which enable to pass to the limit after extraction as in Section \ref{sec:aubin}. For this we update the estimates in Section \ref{sec:M<1} with the key information that the entropy $\bH$ is uniformly bounded. The identification of the limit requires more information concerning the behaviour of the density at infinity. We use the fact that the first momentum drives the evolution of the second one. Finally we conclude that the entropy converges to 0 along some subsequence. Since it is non-increasing, it converges to 0 globally.

\paragraph{A-priori bound} We cannot follow the strategy developped in Section \ref{sec:M<1} since we crucially used $M<1$.
We need to gain some control on the dissipation (\ref{eq:dissip ent M=1}) which is the competition of two opposite contributions that are nearly equally balanced. For that purpose we introduce the function $\Lambda: \R _+ \to \R _+$ such that $\Lambda(0) = 0$ and $\Lambda ' (u)= \left( \log u\right) ^{1/2}_+$. It is non-decreasing, convex and superlinear. Thus there exists $A\in R$ such that $\Lambda(u)^2\geq 2 C_0  u^2 $ for all $u\geq A$. Adapting (\ref{ineq:trace:0}) to our context we get 
\begin{eqnarray} 
\Lambda(n(t,0))^2 &=& \left(-\int_{z>0} \partial _z \left(\Lambda (n(t,z)) \right) \D z\right) ^2 \nonumber \\
&=& \left(-\int_{z>0}   \Lambda' (n(t,z)) n(t,z)  \partial _z (\log n(t,z))  \D z\right) ^2 \nonumber \\
&\le &   \left(\int_{z>0} n(t,z) |\Lambda' (n(t,z))| ^2  \D z\right) \left(\int_{z>0} n(t,z)\left( \partial_z \log n(t,z) \right)^2 \D z  \right)\nonumber\\
& \leq &\left( \int_{z>0} n(t,z) (\log n(t,z))_+\D z \right) \left(\int_{z>0} n(t,z)\left( \partial_z \log n(t,z) \right)^2 \D z \right) \nonumber \\
& \leq &C_0 \int_{z>0} n(t,z)\left( \partial_z \log n(t,z) \right)^2 \D z  \, . \label{eq:trace Lambda:ajout:nico} 
\end{eqnarray}
From (\ref{eq:dissip ent M=1}) and (\ref{eq:trace Lambda:ajout:nico}), we deduce that
\[\frac{ \D}{\D t} \int_{z>0} n(t,z) \log n(t,z) \D  z \leq \left\{\begin{array}{ll} 0 & \mbox{if} \quad n(t,0)\leq A \, ,\\  - \frac{\Lambda(n(t,0))^2}{ C_0} + n(t,0)^2 \leq -n(t,0)^2 &  \mbox{if} \quad n(t,0)\geq A \, .\end{array}\right.\] 
We introduce the set $E = \{t: n(t,0) \geq A\}$. We have obtained the estimate
\begin{equation} 
\int_E n(t,0)^2\D t \leq \int_{z>0} n_0(z) \log n_0(z) \D  z\,, \label{eq:L2 n(t,0):ajout:nico}
\end{equation}  
thus $n(t,0)$ cannot be too large (in $L^2$ sense). 

We deduce from \eqref{eq:dissip ent M=1} and \eqref{eq:L2 n(t,0):ajout:nico} that $\int _0 ^t\int_{z>0} n(s,z)  \left(\partial _z\log n(s,z)\right)^2 \D z \D s$ is bounded for all $t \in (0,T)$. The previous statements prove that Proposition \ref{apriori} remains valid in the case $M=1$. Next, the existence proof, in the case $M=1$, is similar to the case $M<1$ and we do not repeat it here.


\paragraph{Passing to the  limit}

Let $N$ be any integer. We translate the solution in time: we define $u_N(s,x) = n(N + s,x)$. The function $\bH(t)$ is non-increasing and bounded below by zero. Therefore the entropy dissipation (\ref{eq:dissip ent M=1}) converges to zero in an averaged sense. The estimate 
\[ \int_N^{N+1} \left( \int_{z>0} n(t,z)\left( \partial_z \log n(t,z) \right)^2 \D  z   -    n(t,0) ^2 \right)\D t = \bH(N)- \bH(N+1)  \xrightarrow[N\to \infty]{} 0\, ,  \]
reads 
\begin{equation} 
\int_0^{1} \left( \int_{z>0} u_N(s,z)\left( \partial_z \log u_N(s,z) \right)^2 \D  z   -    u_N(s,0) ^2 \right)\D s  \xrightarrow[N\to \infty]{} 0 \, . \label{eq:dissipation u0} 
\end{equation}
We deduce from (\ref{eq:L2 n(t,0):ajout:nico}) that $u_N(s,0)$ is bounded in $L^2(0,1)$ uniformly w.r.t. $N$. Hence both terms are bounded in (\ref{eq:dissipation u0}).
This enables to pass to the limit as in Section \ref{sec:aubin}. Up to extracting a subsequence (labelled with $N'$) there exists $u_\infty$ such that $u_{N'}\to u_\infty$ strongly in $L^2(0,1;B)$:
\begin{equation} \int_0^1 \|u_{N'}(s) - u_\infty(s)\|_{B}^2\D  s \to 0\, . \label{eq:limit} \end{equation}
We can pass to the limit in each term of the averaged dissipation:
\begin{eqnarray*} 
0 = \liminf_{N'\to \infty}\int_0^{1} \left( \int_{z>0} u_{N'}(s,z)\left( \partial_z \log u_{N'}(s,z) \right)^2 \D  z   -    u_{N'}(s,0) ^2  \right)  \D s \\ \geq \int_0^{1} \left( \int_{z>0} u_{\infty}(s,z)\left( \partial_z \log u_{\infty}(s,z) \right)^2 \D  z   -    u_{\infty}(s,0) ^2 \right) \D s  \geq 0 \, . 
\end{eqnarray*}
We have used the $L^2(0,T;L^\infty(\R_+))$ strong convergence (\ref{eq:limit}) to pass to the limit in the nonlinear term $u_{N'}(s,0)^2$, and also the convexity of the functional $\int_{z>0} f(z) \left( \partial _z \log f(z) \right)^2 \D z$ (see Section \ref{sec:aubin}).

\paragraph{Identification of the limit}

We deduce that $u_\infty$ satisfies almost everywhere
\[ u_\infty(s,z) = \beta(s) \exp( - \alpha(s) z)\,,\quad \alpha(s),\beta(s) >0\, . \]
To determine $\alpha(s)$ and $\beta(s)$ we shall use the conservations of mass and first momentum. Since the first momentum is uniformly bounded, we have that $M = \lim \int_{z>0} u_{N'}(t,z)\D z = \int_{z>0} u_{\infty}(t,z) \D z$. This yields $\alpha(s) = \beta(s)$. 

We have proved so far that we can always extract a subsequence such that $u_{N'}(s,z)$ approaches $u_\infty(s,z)$ in $L^2(0,T;B)$. 
We explain below why it is delicate to derive $\alpha(s) = \alpha = \bJ(0)^{-1}$ without any better control of the density $n(t,z)$ as $z\to +\infty$. Suppose we have $\alpha(s)\equiv \overline \alpha$ and the convergence $u_N(s,z)\to u_\infty(z)$ is uniform. We would have on the one hand, 
\begin{equation} 
\alpha^{-1} = \liminf \int_{z>0} z u_{N'}(t,z)\D z \geq  \int_{z>0} z u_{\infty}(z) \D z = (\overline\alpha)^{-1} \, , 
\label{eq:weak convergence moment}
\end{equation}
and on the other hand,
\[  
0\leq \lim \bH(t) = \int_{z>0} u_\infty(z) \log u_\infty(z)\D z + 1 - \log \alpha = \log\overline\alpha  - \log\alpha \, .
\]
We would deduce $\overline \alpha\geq\alpha$ which is the same as (\ref{eq:weak convergence moment}).

In the case $\int_{z>0} z^2 n_0(z)\D z < +\infty$, let us examinate the evolution of the second momentum. We simply have
\begin{equation}
\frac{1}{2}\frac {\D}{\D t} \int_{z>0} z^2 n(t,z)\D z  = M - n(t,0) \bJ(t) = 1 - n(t,0) \alpha^{-1}\, . \label{eq:evol I}
\end{equation}
The idea is to pass to the pointwise limit $n(t,0)\to \overline \alpha$. If $\overline \alpha> \alpha$, the right-hand side of (\ref{eq:evol I}) becomes asymptotically $1 - \overline \alpha\alpha^{-1} < 0$ which leads to a contradiction.   

Let introduce the notation $\bI(t) = \int_{z>0} (z^2/2) n(t,z)\D z$. 
We have 
\[ \bI(N+1) - \bI(N) = \int_N^{N+1} \left(1 - n(t,0)\alpha^{-1}\right) \D t = \int_0^{1} \left(1 - u_N(s,0)\alpha^{-1}\right) \D s\, .  \]
Since $\bI$ is a non-negative quantity, we clearly have
$ \limsup \bI(N+1) - \bI(N) \geq 0$.
Furthermore we have 
$ \limsup \bI(N+1) - \bI(N) \leq 0$.
To see this, assume on the contrary that $ \limsup \bI(N+1) - \bI(N) = \delta >0$. We can extract a converging subsequence. Keeping the same notations as above, we have $\lim \bI(N'+1) - \bI(N') = \delta$. We can pass to the limit similarly (up to further extracting) in the following average quantities:
\begin{eqnarray}
\alpha^{-1} &=& \liminf \int_0^1 \int_{z>0} z u_{N'}(t,z)\D z \D s \nonumber \\
&\geq & \int_0^1  \int_{z>0} z u_{\infty}(s,z) \D z \D s= \int _0^1(\alpha(s))^{-1}\D s\, , \label{eq:lim:crit:1} \\
\delta &=& \lim \int_0^{1} \left(1 - u_{N'}(s,0)\alpha^{-1}\right) \D s = 1 - \alpha^{-1} \int _0^1 \alpha(s)\D s\, . \label{eq:lim:crit:2}
\end{eqnarray}
Inequality (\ref{eq:lim:crit:1}) yields $\int_0^1 \alpha(s)\D s \geq \alpha $ by Jensen's inequality. This is in contradiction with (\ref{eq:lim:crit:2}).

We conclude that $\limsup \bI(N+1) - \bI(N) = 0$. We extract a converging subsequence, such that $\lim \bI(N'+1) - \bI(N') = 0$. 
Hence we obtain (\ref{eq:lim:crit:1}) and (\ref{eq:lim:crit:2}) with $\delta = 0$. The equality case in Jensen's inequality yields $\alpha(s)\equiv \alpha$.

\paragraph{Asymptotic convergence (without any extraction)}

We have proved that there exists a subsequence such that $u_{N'}$ converges towards $u_\infty  = h_\alpha $ in $L^2(0,T;B)$.
We cannot pass to the limit pointwise in time from $L^2$ convergence. However there exists a sequence of times $s_{N'}\in (0,1)$ such that $\|u_{N'}(s_{N'}) - u_\infty\|_{B}\to 0$. This includes uniform convergence and uniform decay at infinity. We can pass to the limit in the entropy term and we obtain $H(u_{N'}(s_{N'})) \to H(u_\infty) = 0$. It means 
$H[n(N' + s_{N'})] \to 0$. Using the non-increasing property of the entropy we have $H[n(t)] \to 0$ as $t\to \infty$ (without extracting any subsequence). 

Finally we recall the Csiszar-Kullback inequality  \cite{Csiszar,Kullback}. For any non-negative functions $f,g \in L^1(\R_+)$ such that $\int_{x>0} f(x) \D x=\int_{x>0} g(x) \D x=1$, the following inequality holds true, 
\[  \|f-g \|^2_{L^1(\R_+)}\le 4 \int_{x>0} f(x)\log \left(\frac{f(x)}{g(x)}\right) \D x \, . \]
This yields $\|n(t) - h_{\alpha}\|_{L^1}\to 0$.

\subsubsection{Self-similar decay in the sub-critical case}

\label{sec:self-similar}

In the sub-critical case $M<1$ the density $n(t,z)$ is expected to decay with a self-similar diffusion scaling \cite{BDP}. To catch this asymptotic behaviour we rescale the density accordingly:
\[ n(t,z) = \frac{1}{\sqrt{1+2t}} u\left( \log\sqrt{1+2t}  , \frac{z}{\sqrt{1+2t}}\right)\, . \]
The new density $u(\tau,y)$ satisfies:
\begin{equation}\label{eq:rescaled}
\partial_\tau u(\tau,y) = \partial_{yy} u(\tau,y)   + \partial_y \left(y u(\tau,y) \right) + u(\tau,0)\partial_y u(\tau,y) \, ,
\end{equation}
together with a no-flux boundary condition: $\partial _y u (\tau,0)+u(\tau,0)^2 =0$.
The additionnal left-sided drift contributes to confine the mass in the new frame $(\tau,y)$. The unique stationary equilibrium in this new setting can be computed explicitely: 
\begin{equation} \label{eq:stat state rescaled}
G_\alpha(y) = \alpha\exp\left(-\alpha y - {y^2}/2\right)\, ,
\end{equation}  
where $\alpha$ is uniquely defined by the condition $\int_{y>0} G_\alpha(y)\D y = M$. This rewrites $P(\alpha) = M$, $P$ being an increasing function defined as follows:
\[ P(\alpha) = \int_{y>0}  \exp\left(-  y - \frac{y^2}{2\alpha^2}\right)\D  y \, , \quad \left\{\begin{array}{r} \lim_{\alpha \to 0} P(\alpha) = 0 \\ \lim_{\alpha \to +\infty} P(\alpha) = 1 \end{array}\right. \,   . \]
We re-define the relative entropy and the first momentum in the 
rescaled frame:
\begin{eqnarray*} 
\bH(\tau) &=& \int_{y>0}\frac{u(\tau,y)}{G_\alpha (y)} \log \left(\frac{u(\tau,y)}{G_\alpha (y)}\right)G_\alpha (y)\D  y\, ,\\
\bJ(\tau) &=& \int_{y>0} y u(\tau,y)\D y\, . 
\end{eqnarray*} 
We also introduce a Lyapunov functional for equation (\ref{eq:rescaled}):
\[ \bL(\tau)=\bH(\tau)  + \frac{1}{2(1-M)}  \left( \bJ(\tau) - \alpha(1-M) \right)^2 \, . \]
Note that it is a non-negative quantity by Jensen's inequality.

\begin{lemma} \label{lem:entropy u}
The Lyapunov functional $\bL$ is non-increasing:
\begin{equation*}
\frac{\D}{\D t} \bL(\tau) = - \bD(\tau) \leq 0\, . 
\end{equation*}
The dissipation reads as follows
\begin{eqnarray}
 \bD(\tau) &=& \int_{y>0} u(\tau,y) \left(   \partial_y \log u(\tau,y) + y + u(\tau,0) \right)^2\D  y \nonumber \\
 & & + \frac{1}{(1-M)}\left(\frac{\D }{\D \tau} \bJ(\tau)\right)^2\, . 
 \label{eq:dissipation u}
 \end{eqnarray}
\end{lemma}

\begin{proof}
We compute the evolution of the entropy as previously:
\begin{eqnarray}
\frac{\D }{\D \tau} \bH(\tau)&=& \int_{y>0} \partial _\tau u(\tau,y)\left( \log (u(\tau,y))  +\alpha y + \frac{y^2}{2} \right) \D  y \nonumber \\
&=& - \int_{y>0}  \left( \partial_y u(\tau,y) + u(\tau,0) u(\tau,y) + y u(\tau,y) \right) \left( \frac{ \partial_y u(\tau,y)}{u(\tau,y)}  +\alpha  + y \right) \D  y \nonumber \\
& =&  - \int_{y>0} u(\tau,y) \left( \partial_y \log u(\tau,y) + y  \right)^2\D  y +
u(\tau,0)^2 - u(\tau,0) \bJ(\tau) \nonumber \\
& & \qquad +\alpha u(\tau,0) - \alpha \bJ(\tau)  - \alpha u(\tau,0) M  \nonumber\\
& = & - \int_{y>0} u(\tau,y) \left( \partial_y \log u(\tau,y) + y +u(\tau,0) \right)^2\D  y \nonumber\\
& &\qquad +
(M-1)u(\tau,0)^2 + u(\tau,0) \bJ(\tau) +\alpha (1-M) u(\tau,0) - \alpha \bJ(\tau) \, . \label{eq:u0}
\end{eqnarray}
Moreover, the time evolution of the first momentum reads in the rescaled frame:
\[\frac{\D }{\D \tau}\bJ(\tau)  = (1-M)u(\tau,0) - \bJ(\tau)\, .\]
As compared to (\ref{eq:moment0}) the additional contribution is due to the rescaling drift. We can eliminate $u(\tau,0)$ from (\ref{eq:u0}) in the two following steps:
\begin{eqnarray}
u(\tau,0) \bJ(\tau) +\alpha (1-M) u(\tau,0) - \alpha \bJ(\tau)   
& = &
\frac{\bJ(\tau)}{(1-M)}\frac{\D }{\D \tau} \bJ(\tau)+ \frac{\bJ(\tau)^2 }{(1-M)}+ \alpha \frac{\D }{\D \tau} \bJ(\tau) \nonumber \\
& =& -\frac{\D }{\D \tau}  \frac{\left( \bJ(\tau) - \alpha(1-M) \right)^2}{2(1-M)}  
\nonumber\\ 
& & 
+ \frac{2\bJ(\tau)}{(1-M)}\frac{\D }{\D \tau} \bJ(\tau)+ \frac{\bJ(\tau)^2}{(1-M)}  \, , 
\label{eq:entropy J}
\end{eqnarray} 
and
\begin{equation} 
- \frac{1}{(1-M)}\left(\frac{\D }{\D \tau} \bJ(\tau)\right)^2   
=(M-1)u(\tau,0)^2 +\frac{2\bJ(\tau)}{(1-M)}\frac{\D }{\D \tau} \bJ(\tau)+ \frac{\bJ(\tau)^2 }{(1-M)}
 \, .  \label{eq:dissipation J}
\end{equation}
Combining (\ref{eq:u0}) -- (\ref{eq:entropy J}) -- (\ref{eq:dissipation J}) the proof of Lemma \ref{lem:entropy u} is complete.
\end{proof}

To prove convergence of $u(\tau,\cdot)$ towards $G_\alpha$ we develop the same strategy as in Section \ref{sec:M=1} for the critical case $M = 1$. The main argument (apart from passing to the limit) consists in identifying the possible configurations $u_\infty$ for which the dissipation $\bD$ vanishes. In fact this occurs if and only if both terms in (\ref{eq:dissipation u}) are zero. This means that $\bJ_\infty(\tau) = (1 - M) u_\infty(\tau,0)$ on the one hand, and on the other hand,
\[ \partial_y \log u_\infty(\tau,y)  + y + u_\infty(\tau,0) = 0  \, .\]
We obtain that $u_\infty \equiv G_\alpha$, where $G_\alpha$ is given by (\ref{eq:stat state rescaled}). To pass to the limit as in Section  \ref{sec:M=1} we need to gain some good control of $\int_{y>0} u(\tau,y) \left(   \partial_y \log u(\tau,y) \right)^2\D y$ from the dissipation term $\bD$. The situation here is simpler than in Section \ref{sec:M=1} since the mass is sub-critical. The argument goes as follows
\begin{eqnarray*}
& &\int_{y>0} u(\tau,y) \left(   \partial_y \log u(\tau,y)  + y + u(\tau,0) \right)^2\D  y
\\ &= &
\int_{y>0} u(\tau,y) \left(   \partial_y \log u(\tau,y) \right)^2\D y + (M-2) u(\tau,0)^2  + 2 u(\tau,0) \bJ(\tau )\\
 & & + \int_{y>0} y^2 u(\tau,y)\D y - 2M \\
& \geq &  \left(M + \frac{1}{M}-2\right) u(\tau,0)^2   - 2M \, ,
\end{eqnarray*}
where we have used inequality (\ref{ineq:trace:0}). The quantity $M + M^{-1} - 2$ is positive since $M<1$. Hence, recalling Proposition \ref{apriori} we can prove directly that $u(\cdot,0)$ belongs to $L^2$ locally in time (this was the purpose of (\ref{eq:trace Lambda:ajout:nico}) -- (\ref{eq:L2 n(t,0):ajout:nico})). 

Finally, we obtain that $\bL$ converges to zero as $\tau\to+\infty$. So $u(\tau,\cdot)$ converges towards $G_\alpha$ in entropy sense.

\subsection{Blow-up of solutions for super-critical mass}\label{sec:BU}

To prove that solutions blow-up in finite time when mass is super-critical $M>1$ and $n_0$ is non-increasing, we show that the first momentum of $n(t,z)$ cannot remain positive for all time. This technique was first used by Nagai \cite{Nagai}, then by many authors in various contexts (see \cite{Biler95,BilerWoyczynski,Corrias.Perthame.Zaag,DP,CieslakLaurencot} for instance).

The assumption that $n_0$ is a non-increasing function guarantees that $n(t,\cdot)$ is also non-increasing  for any time $t>0$ due to the maximum principle. In fact the derivative $v(t,z) = \partial_z n(t,z)$ satisfies a parabolic type equation without any source term, it is initially non-positive, and it is non-positive on the boundary due to (\ref{cl1D}).

Therefore $-\partial_z n(t,z)/n(t,0)$ is a probability density at any time $t>0$. We deduce from the Jensen's inequality the following interpolation estimate:
\[\left(\int_{z>0} z \frac{-\partial_z n(t,z)}{n(t,0)}\D  z\right)^2  \leq 
\int_{z>0} z^2 \frac{-\partial_z n(t,z)}{n(t,0)}\D  z\, .\]
It rewrites in a more convenient way as follows,
\begin{equation}\label{interpol}
M^2  \leq 2 n(t,0)  \int_{z>0} z   n(t,z) \D  z \, .
\end{equation}

We denote the first momentum  $\bJ(t) = \int_{z>0} z n(t,z)\D  z$. We plug (\ref{interpol}) into the evolution of the moment (\ref{faible11}):
\begin{eqnarray}
 \bJ(t)& =& \bJ(0) + (1 - M) \int_0^t n(s,0)\D s  \nonumber \\
 & \leq &\bJ(0) + \frac{(1 - M)M^2}{2} \int_0^t \frac{1}{\bJ(s)}\D s\label{eq:moment0}\, .
\end{eqnarray}
We introduce the auxiliary function $\bK(t) = \bJ(0) + (1 - M)M^2 \int_0^t \bJ(s)^{-1}\D s$. It is positive and it satisfies the following differential inequality:
\[\frac{\D }{\D t}\bK(t)  = \frac{(1 - M)M^2}{2} \frac{1}{\bJ(t)}   \leq \frac{(1 - M)M^2}{2} \frac{1}{\bK(t)} \, ,\]
hence,
\[\frac{\D }{\D t}\bK(t)^2  \leq  (1 - M)M^2\, . \]
We obtain a contradiction: the maximal time of existence $T^*$ is necessarily finite when $M>1$. On the other hand, following \cite{JL}, it can be proved  that the modulus of integrability has to become singular at $T^*$:
\[ \lim_{K\to +\infty} \left( \sup_{t\in(0,T^*)} \int_{z>0} (n(t,z)-K)_+\D z\right) >0\, . \]
Otherwise a truncation method enables to prove local existence by replacing $n$ with $(n - K)_+$ for $K$ sufficiently large.

\begin{remark}
It is natural to perform the Laplace transform on the equation (\ref{eq:1D}) $\mathcal L_z(n(t,z)) = \hat n(t,\zeta)=\int_{z>0} n(t,z) \exp(-\zeta z) \D z$. Then the occurence of blow-up is clear after transformation. We refer the reader to \cite{Calvez.Carrillo2010} where the Fourier transformation has been applied successfully to analysing a one-dimensional caricature of the two-dimensional Keller-Segel equation.
\end{remark}

\section{Variants of blow-up criteria}\label{secvariants}

In this section we determine necessary conditions for blow-up to occur for a fast decaying interaction potential (Section \ref{finite_range}) and for a finite interval (Section \ref{bounded}).

\subsection{Finite range of action}\label{finite_range}

In this part we consider the following system:
\begin{equation}\label{eqfini:fr}
\partial _t n(t,z) = \partial _{zz} n(t,z) - \partial_z\left(n(t,z) \partial_z \phi(t,z)\right)\, ,  \quad t>0\, ,  \, z\in (0,+\infty)\, , 
\end{equation}
with zero-flux at $z = 0$ and the attractive potential is given by
\begin{equation} \label{eq:phi alpha}
-\partial_{zz} \phi (t,z) + \alpha^2 \phi (t,z) = 0\, , \quad - \partial_z \phi (t,0) = n(t,0)\, .
\end{equation} 
We introduce the exponential moment of the solution:
\[  \bJ_\alpha(t) = \int_{z>0} \exp(\alpha z) n(t,z)  \D z\, . \]
\begin{proposition} \label{prop:1D:FR}
Assume $M>1$ and the exponential moment is small in the sense of criterion (\ref{eq:BU alpha}) below. Assume in addition that $\exp(-\alpha z) n_0(z)$ is a non-increasing function. Then the solution to (\ref{eqfini:fr}) -- (\ref{eq:phi alpha}) with initial data $n(0,z) = n_0(z)$ blows-up in finite time.
\end{proposition}
\begin{proof}
The attractive field is given by $\partial_z \phi(t,z) = -\exp(-\alpha z) n(t,0)$.
Similarly to the proof of Theorem \ref{th:1D BU}, we compute the time derivative of $ \bJ_\alpha(t)$: 
\[ \frac{\D }{\D t} \bJ_\alpha (t)  =  \alpha ^2 \bJ_\alpha(t) + \alpha n(t,0)(1-M)\, . \]
We check that the function $u(t,z) = \exp(-\alpha z )n(t,z)$ is decreasing w.r.t. $z$ for all time $t>0$. For this purpose we write the equation for $v(t,z) = \partial_z u(t,z)$. This reads as follows
\begin{eqnarray*}  
\partial_t u(t,z) &=& \partial_{zz} u(t,z) + 2 \alpha \partial_{z} u(t,z) + \alpha^2 u(t,z) + \exp(-\alpha z) n(t,0) \partial_z u(t,z)\, , \\
\partial_t v(t,z) &=& \partial_{zz} v(t,z) + 2 \alpha \partial_{z} v(t,z) + \alpha^2 v(t,z)  + \exp(-\alpha z) n(t,0) \partial_z v(t,z) \\
& &- \alpha \exp(-\alpha z) n(t,0) v(t,z) \, .
\end{eqnarray*} 
Since the boundary condition reads $v(t,0) = -\alpha n(t,0)- n(t,0)^2\leq 0$ and the above parabolic equation preserves non-positivity we deduce that $v(t,z)\leq 0$ if $v(0,z)\leq 0$.

We can adapt the inequality (\ref{interpol}) to the function $u(t,z)$ and we obtain
\begin{eqnarray*} 
M^4 &\leq &\left(\int_{z>0} \exp(\alpha z) n(t,z)  \D  z\right)^2 \left(\int_{z>0} u(t,z)  \D  z\right)^2  \\
&\leq &  \bJ_\alpha(t)^2 n(t,0)^2   \left(\int_{z>0} z \frac{-\partial_z u(t,z)}{u(t,0)} \D  z\right)^2\\
& \leq & \bJ_\alpha(t)^2 n(t,0)^2 \int_{z>0} \left(\frac{\exp(2\alpha z) - 1 - 2\alpha z}{2\alpha^2}\right) \frac{-\partial_z u(t,z)}{u(t,0)} \D  z \\ 
& \leq &\frac{1}{\alpha} \bJ_\alpha(t)^2 n(t,0) \int_{z>0} \left( \exp (\alpha z)- \exp (-\alpha z)  \right) n(t,z)\D  z  \\
& \leq &  \frac{1}{\alpha} \bJ_\alpha(t)^2 n(t,0) \left(   \bJ_\alpha(t) - \frac {M^2}{ \bJ_\alpha(t)}\right) \, .    
\end{eqnarray*}
Finally, when $M>1$ we obtain that:
\begin{eqnarray*}
\frac{\D }{\D t}\bJ_\alpha(t) & \le  &\alpha^2  \bJ_\alpha(t) + \frac{\alpha^2(1 - M)M^4}{  \bJ_\alpha(t)^3 \left( 1 - \frac{M^2}{\bJ_\alpha(t)^2}\right) } \, . 
\end{eqnarray*}
Notice that $ \bJ_\alpha(0) >M$ by definition. We get an obstruction to global existence if the following condition holds true,
\begin{equation}\label{eq:BU alpha}
 \frac{ \bJ_\alpha(0)^4}{M^4} \left( 1 - \frac {M^2}{ \bJ_\alpha(0)^2} \right) < (M-1) \, .
\end{equation} 
\end{proof}

\subsection{Finite interval}\label{bounded}

In this part we consider the equation (\ref{eq1D}) on a finite interval $(0,L)$ for some $L>0$, namely,
\begin{equation}\label{eqfini}
\partial _t n(t,z) = \partial _{zz} n(t,z) + (n(t,0)-n(t,L)) \partial _z n(t,z)\, , \quad t >0\, , \, x\in (0,L)\, ,
\end{equation}
together with $n(t=0,z) = n_0(z)\geq 0$ and zero-flux boundary conditions at both sides of the interval.

Equilibrium configurations are given by the family of functions:
\begin{equation} \label{eq:h (0,L)}
h(z) = \alpha \exp(- (\alpha - \beta)z )\, , \quad \beta = \alpha \exp(- (\alpha - \beta)L)\, . 
\end{equation}
There are two possibilities, either $\alpha = \beta$ and $h$ is constant, or $\alpha \neq \beta$ and $M = \int_0^L h(z) \D z = 1$. Observe that given $\alpha >0$ there exists a unique $\beta$ satisfying (\ref{eq:h (0,L)}). If $\alpha L < 1$ then $\beta >\alpha$ ($h$ is increasing), whereas if $\alpha L >1$ then $\beta < \alpha$ ($h$ is decreasing).\\

\begin{proposition}
\label{faibleBDP1:finite}
Assume $M>1$ and the first moment is small: $4\bJ(0) < L M$. Assume in addition that $ n_0(z)$ is a non-increasing function. Then the solution to (\ref{eqfini}) with initial data $n(0,z) = n_0(z)$ blows-up in finite time.\\
 \end{proposition} 
\begin{proof}
We proceed again as in the proof of theorem (\ref{th:1D BU}).
From Jensen's inequality, it follows that:
\[\left(\int_{0}^L z \frac{-\partial_z n(t,z)}{n(t,0)-n(t,L)}\D  z\right)^2  \leq 
\int_{0}^L z^2 \frac{-\partial_z n(t,z)}{n(t,0)-n(t,L)}\D  z\, , \]
hence, using that $n(t,0)>n(t,L)$ for any time $t>0$, we deduce that
\begin{equation}\label{interpolfini}
(M-Ln(t,L))^2  \leq  (n(t,0)-n(t,L)) \Big( 2\int_{0}^L z   n(t,z)  \D  z-L^2 n(t,L)\Big)  \, ,
\end{equation}
and the inequality remains true when $n(t,0)=n(t,L)$ and $n(t,\cdot)$ is constant.
Therefore, the first momentum  $\bJ(t) = \int_{0}^L z n(t,z)\D  z$ satisfies:
\begin{eqnarray*}
\frac{\D }{\D t} \bJ(t)  & =& (1-M)(n(t,0) -  n(t,L))  \\
& \leq &(1-M) \frac{(M - Ln(t,L))^2}{2\bJ(t) - L^2 n(t,L)} \\
& \leq &(1-M) \frac{M^2 - 2M L n(t,L)}{2\bJ(t)}\, .
\end{eqnarray*}
On the other hand, from (\ref{interpolfini}) again, it follows that $2\bJ(t) \ge L^2 n(t,L)$ and we deduce that
\[\frac{\D }{\D t} \bJ(t)  \leq  \frac{M(1-M)}{2\bJ(t)} \left(M - \frac{4\bJ(t)}{L} \right)\, ,\] 
and the result follows by contradiction as in Section \ref{sec:BU}.
\end{proof}

\section{The model with with dynamical exchange of markers at the boundary: prevention of blow-up and asymptotic behaviour} \label{sec:ODE/PDE}

In Section \ref{sec:BU}, we proved that finite blow-up occurs in the basic model (\ref{eq:1D}) when mass is super-critical $M>1$. On the other hand the model which was originally proposed in \cite{HBPV} is the following:
\[\left\{\begin{array}{l} 
\partial _t n (t,z)=  \partial _{zz} n (t,z) +   \mu (t) \partial _z n (t,z) \, , \quad t >0\, , \, z\in (0,+\infty) \medskip \\ 
\frac{\D }{\D t}\mu(t)=   n(t,0)-  \mu(t)\, ,
\end{array}\right.\]
together with the flux condition at the  boundary:
\begin{equation} 
  \partial _z n (t,0)+  \mu(t) n (t,0)=
\frac{\D }{\D t}\mu(t)\, . \label{eq:BC2dim}
\end{equation}
The quantity $\mu$ represents the concentration of markers which are sticked to the boundary and thus create the attracting drift. 
The dynamics of $\mu$ is driven by simple attachment/detachment kinetics. The mass of molecular markers is shared between the free particles $n(t,z)$ and the particles on the boundary $\mu(t)$. The boundary condition (\ref{eq:BC2dim}) guarantees conservation of the total mass:
\begin{equation} 
\int_{z>0} n(t,z)\D z + \mu(t) = M\, . \label{eq:mass conservation dim}\end{equation} 
From (\ref{eq:mass conservation dim}), we easily deduce that finite time blow-up cannot occur since the drift $\mu(t)$ is bounded by $M$.  
We denote by $m(t)$ the mass of free particles: 
\begin{equation*}
m(t) = \int_{z>0} n(t,z)\D z\, .
\end{equation*}
The conservation of mass reads 
\[
\frac{\D }{\D t}m(t) + \frac{\D }{\D t}\mu(t) = 0\, .\]
We re-define the relative entropy as follows:
\[\bH(t)= \int_{z>0}\frac{n(t,z)}{m(t)h  (z)} \log \left(\frac{n(t,z)}{m(t)h(z)}\right)h (z)\D  z\, ,\]
where the asymptotic profile $h$ is given by:
\[h(z) = \nu \exp\left(-\nu z\right)\, , \quad \nu = M-1\, . \]
When mass is super-critical $M>1$, we shall prove that the density of free markers $n(t,z)$ converges in relative entropy towards  $h$, whereas the concentration of markers sticked at the boundary $\mu(t)$ converges to $\nu$. 
This is achieved using a suitable Lyapunov functional as in Sections \ref{sec:M=1} and \ref{sec:self-similar}. We introduce accordingly
\[\bL(t) = m(t)\bH(t)   + \frac{1}{2}\left(\mu(t) - \omu\right)^2 + \mu(t)\log\left(\frac{ \mu(t)}{\nu}\right)   + m(t)\log m(t)  \, .\]
The rest of this Section is devoted to the proof of the following Lemma.\\

\begin{lemma}\label{eq:dissipation mu}
The Lyapunov functional $\bL$ is non-increasing:
\[\frac{\D}{\D t} \bL(t ) = - \bD(t ) \leq 0\, . \]
The dissipation reads as follows
\begin{eqnarray*}
\bD(t)  &=& \int_{z>0} n(t,z) \left(\partial _z\log n(t,z) + \frac{n(t,0)}{m(t)}\right)^2\D  z + m(t)\left(\frac{n(t,0)}{m(t)} -   \mu(t) \right)^2\\
& &+   \left( n(t,0) -  \mu(t) \right) \log \left(\frac{n(t,0)}{\mu(t)}\right)  + \mu(t) \left( \mu(t)- \omu \right)^2\, .
\end{eqnarray*}
\end{lemma}

\begin{proof}
We compute below the time evolution of the relative entropy. This is strongly inspired from the previous computation, but this takes into consideration the non-conservation of mass for the free markers density and the additional dynamics of $\mu$.
\begin{eqnarray*}
\frac{\D }{\D t} \left( m(t)\bH(t)\right)
& =& \frac{\D }{\D t}\int_{z>0} n(t,z) \left( \log\left(\frac{n(t,z)}{m(t)}\right) -\log \omu +\omu z\right)\D  z\\
&= &\int_{z>0} \partial _t\left( n(t,z) \right)\left( \log\left(\frac{n(t,z)}{m(t)} \right) -\log \omu +\omu z\right)\D  z \\
& &\qquad+\int_{z>0} n(t,z)\,  \partial _t \, \log  \left(\frac{n(t,z)}{m(t)}\right)\D  z\\
&= & \int_{z>0}\partial _z\left( \partial _z n(t,z) +\mu(t) n(t,z) \right) \left( \log\left(\frac{n(t,z)}{m(t)} \right)  +\omu z\right)\D  z\\
& &\qquad- \frac{\D }{\D t} m(t) \log \omu \, ,
\end{eqnarray*}
where we have used the identity
\[ \int_{z>0} n(t,z)\,  \partial _t \, \log  \left(\frac{n(t,z)}{m(t)}\right)\D  z= m(t) \int_{z>0} \partial_t\left(\frac{n(t,z)}{m(t)}\right)\D  z = 0\, .\]
We integrate by parts to get
\begin{eqnarray*}
\frac{\D }{\D t} \left( m(t)\bH(t)+m(t)\log \omu \right)
&=& - \int_{z>0}(\partial _z n(t,z) +\mu(t) n(t,z)) \left( \frac{\partial _z n(t,z)}{n(t,z)}    +\omu \right)\D  z\\
& &\qquad - \left(\partial _z n (t,0)+\mu(t) n(t,0)\right) \log\left(\frac{n(t,0)}{m(t)} \right)\\
&=& -  \int_{z>0} n(t,z) \left(\partial _z\log n(t,z)\right)^2\D  z + (\omu +  \mu(t) ) n(t,0)\\
& &\qquad -  m(t)\mu(t) \omu -   \log\left(\frac{n(t,0)}{m(t)} \right)\frac{\D }{\D t} \mu(t)\, .
\end{eqnarray*}
We use again the following key identity
\begin{eqnarray*}
\int_{z>0} n(t,z) \left(\partial _z\log n(t,z)\right)^2\D z =    
\int_{z>0} n(t,z) \left(\partial _z\log n(t,z) + \frac{n(t,0)}{m(t)}\right)^2\D  z +  \frac{n(t,0)^2}{m(t)}\, . 
\end{eqnarray*}
We end up with the following expression for the dissipation of the corrected entropy,
\begin{eqnarray*}
\frac{\D }{\D t} \left( m(t)\bH(t) +m(t)\log \omu \right) 
& = &- \int_{z>0} n(t,z) \left(\partial _z\log n(t,z) + \frac{n(t,0)}{m(t)}\right)^2\D  z  \\ 
& &\qquad   - \frac{ n(t,0)^2}{m(t)}   + (\omu +  \mu(t))  n(t,0)  -  m(t) \mu(t) \omu \\
& &\qquad -    \log\left(\frac{n(t,0)}{m(t)}\right)\frac{\D }{\D t}\mu(t)\, .
\end{eqnarray*}
On the first hand, we have that
\begin{eqnarray*}
- \frac{ n(t,0)^2}{m(t)}   + (\omu +  \mu(t))  n(t,0)  -  m(t) \mu(t) \omu
& = &\left( - \frac{ n(t,0)}{m(t)} + \omu \right)\left( n(t,0) - m(t)\mu \right) \\
& = & - m(t)\left(\frac{n(t,0)}{m(t)} -   \mu(t) \right)^2  \\
& & -   ( \mu(t)- \omu)\left( n(t,0) - m(t)\mu(t) \right) \, ,
\end{eqnarray*}
and on the other hand, we see that
\begin{eqnarray*}
& &- \log\left(\frac{n(t,0)}{m(t)}\right)\frac{\D }{\D t}\mu(t)\\
& & \quad =- \log\left(\frac{n(t,0)}{\mu(t)}\right)\frac{\D }{\D t}\mu(t)
- \log\left(\frac{\mu(t)}{m(t)}\right)\frac{\D }{\D t}\mu(t) \\
&  & \quad =
- \left(n(t,0) - \mu(t)\right) \log\left(\frac{n(t,0)}{\mu(t)}\right)  
- \log\left( \mu(t) \right)\frac{\D }{\D t}\mu(t) - \log\left( m(t) \right)\frac{\D }{\D t}m(t)
\\
& & \quad = - \left(n(t,0) - \mu(t)\right) \log\left(\frac{n(t,0)}{\mu(t)}\right) 
\\
& &\qquad 
- \frac{\D }{\D t}\left( \mu(t) \log  \mu(t)   - \mu(t)  + m(t) \log  m(t)    - m(t) - \nu \log \nu + M \right) \, .
\end{eqnarray*}
The last contribution to be reformulated is
\begin{eqnarray*}
 -   ( \mu(t)- \omu)\left( n(t,0) - m(t)\mu(t) \right) 
 & = &- ( \mu(t)- \omu)\left( \frac{\D }{\D t}\mu(t) + (1 - m(t)) \mu(t)  \right)\\
& =& - ( \mu(t)- \omu)\left( \frac{\D }{\D t}\mu(t) + (\mu(t) - \omu)\mu(t) \right) \\
& = &- \frac{1}{2}\frac{\D }{\D t} ( \mu(t)- \omu)^2 - \mu(t) (\mu(t) - \omu) ^2\, .
\end{eqnarray*}
Combining all these calculations we conclude the proof of Lemma \ref{eq:dissipation mu} 
\end{proof}

Following the lines of Section \ref{sec:M=1} we can prove that $\mu(t)$ converges to $\omu$,  the partial mass $m(t)$ converges to 1, and the density $n(t,\cdot)$ converges to the stationary state $h$ as $t\to \infty$. We omit the details.

\section{The higher dimensional case $N\geq 2$}\label{secdimsup}

In this section we investigate the possible behaviours of the equation (\ref{eq:2D model})  in dimension $N\geq 2$ with the two possible choices (\ref{eq:u1}) and (\ref{eq:u2}) for the advection field.  

\subsection{Global existence}

We give the proof of Theorem \ref{thdim2}. Since many of the arguments are similar to the one-dimensional case, we only sketch the proof and focus on the propagation of $L^p$ bounds, which is the crucial {\em a priori} estimate as soon as entropy methods are lacking \cite{JL}.

Let $n$ be a solution of (\ref{eq:2D model}) with $\nabla\cdot \bu \geq 0$ and $\bu(t,y,0)\cdot \be_z = n(t,y,0)$. We see that
\begin{eqnarray}
 \frac{\D }{\D t} \int_{\mH} n(t,x)^p \D x
 &=&  -p \int_{\mH} \nabla  n(t,x)^{p-1} \cdot \nabla n(t,x) \D x  \nonumber
 \\& &  +p \int_{\mH} \nabla  n(t,x)^{p-1} \cdot \mathbf{u}(t,x)\, n(t,x) \D x\, .
 \label{relation2}
\end{eqnarray}
On the first hand, we have that
\[-p \int_{\mH} \nabla  n(t,x)^{p-1} \cdot  \nabla n(t,x) \D x = -\frac{4(p-1)}{p} \int_{\mH}\left| \nabla n(t,x)^{p/2}\right|^2 \D x\, ,\]
and on the other hand,
\[\frac{p}{p-1} \int_{\mH} \nabla  n(t,x)^{p-1} \cdot \mathbf{u}(t,x)\, n(t,x) \D x 
=  - \int_{\mH}   n(t,x)^{p} \left(\nabla\cdot \bu\right)\D x + \int _x  n(t,y,0)^{p+1} \D y \, . 
\]
To estimate the two opposite trends in (\ref{relation2}) we use the following Sobolev trace inequality \cite{Biezuner} and \cite{Nazaret}: there exists a constant $C_r$ such that for any non-negative $f\in W^{1,r}$ we have,
\begin{equation}\label{GS0}
\left(  \int_{y\in \R^{N-1}  } f (y,0)^{r^*} \D y \right)^{1/r^*} \le C_r \left(\int_{\mH}\left| \nabla f(x)\right|^r \D x \right)^{1/r}\, ,
\end{equation} 
where $ r^*=\frac{(N-1)r}{N-r}$.
Applying the previous inequality (\ref{GS0}) with $f=n^s$, we obtain the estimates:
\begin{eqnarray*}
 & &\int_{y\in \R^{N-1} } n(t,y,0)^{r^* s} \D y  
 \le  C_r  \left(\frac{2s}{p}\right)^{r^*} \left(\int_{\mH}\left| \nabla  n(t,x)^{p/2} \, n(t,x)^{s-\frac{p}{2}}\right|^r \D x \right)^{r^*/r}
\\
&  & \qquad \leq C_r \left(\frac{2s}{p}\right)^{r^*} \left(\int_{\mH}\left| \nabla n(t,x)^{p/2} \right| ^2 \D x \right)^{r^*/2}  \left(\int_{\mH} \left( n(t,x)^{s-\frac{p}{2}}\right) ^{\frac{2r}{2-r}} \D x \right)^{\frac{(2-r)r^*}{2r}}\, .
\end{eqnarray*} 
We infer that $L^N$ is the critical space for global existence. Hence we choose
\[\left(s-\frac{p}{2}\right)\frac{2r}{2-r}=N\, .\]
On the other hand, we also require that 
\[1 = \frac{r^*}{2}= \frac{1}{2} \frac{(N-1)r}{N-r}\, .\]
A straightforward computation leads to 
\[r=\frac{2N}{N+1}\, , \quad s=\frac{p+1}{2}\,,\quad r^*s = p+1\, ,\quad \frac{(2-r)r^*}{2r} = \frac{1}{N}\,  .\]
Therefore we deduce that 
\[\frac{\D }{\D t} \int_{\mH}  n(t,x) ^p \D x
\le -\frac{4(p-1)}{p} \left( 1- C  \|n(t)\|_{L^N}\right) \int_{\mH}\left| \nabla n(t,x)^{p/2} \right|^2 \D x 
\, .
\]
The peculiar choice $p = N$
yields global existence  if  $\|n(0)\|_{L^N}$ is smaller than some explicit threshold as in \cite{Corrias.Perthame.Zaag}.

\subsection{Blow-up of solutions in the first case (\ref{eq:u1})}\label{eq:sec:BU u1}

We compute the evolution of the second momentum $\bI(t)= \frac12 \int_{\mH}  |x|^2 n(t,x) \D  x$ as for the classical Keller-Segel system (see \cite{P} and references therein): 
\[\frac{\D \bI(t) }{\D t} 
 = N M 
 -\int_{\mH} z n(t,y,0) n(t,y,z)   \D x\, . 
\]
Next, define $M(t,y) =\int_{z>0} n(t,y,z)\D  z$. Under the assumption $\partial_z n(t,x) \leq 0$ for all $x\in \mH$ and $t>0$, inequality (\ref{interpol}) rewrites 
\[ M(t,y)^2  \leq 2 n(t,y,0)  \int_{z>0} z   n(t,y,z) \D  z \, .
\]
We deduce that:
\begin{eqnarray}\label{relationJ1}
 \frac{\D \bI(t) }{\D t} &\le&  N M  -\frac{1}{2}\|M(t,y)\|^2_{L^2}.
\end{eqnarray}
By interpolation there exists a constant $C$ such that 
\begin{equation}\label{eq:interp2D} 
M^{\frac{N+3}{2}} \leq C \bI(t)^{\frac{N-1}2}\|M(t,y)\|_{L^2}^2\, . 
\end{equation}
Indeed we have
\begin{eqnarray*}
M
&=&  \int _{|y|< R} M(t,y) \D y + \int _{|y|>R} M(t,y) \D y \\
& \le & C   R^{(N-1)/2} \left( \int _{R^{N-1} }  M(t,y)^2 \D y\right)^{1/2} +R^{-2}  \int  _{R^{N-1} }  |y|^2M(t,y) \D y  \\
&\le & C  R^{(N-1)/2} \|M(t,y)\|_{L^2} + R^{-2}  \bI(t)\, .
\end{eqnarray*}
Optimizing with respect to $R$ we get (\ref{eq:interp2D}). Combining (\ref{eq:interp2D}) and (\ref{relationJ1}) we conclude that the solution blows-up in finite time if $\bI(0)\leq C M^{\frac{N+1}{N-1}}$.

\subsection{Blow-up in the second case (\ref{eq:u2})}

We recall the expression of the advection field in the potential case (\ref{eq:u2}):
\[ \bu(t,x) = - \int_{y'\in \R^{N-1}} \frac{(y-y',z)}{\left(|y-y'|^2 + z^2\right)^{N/2}} n(t,y',0)\D  y'\, .  \] 
Therefore we have
\begin{eqnarray*}
 \frac{\D \bI(t) }{\D t} & =& N M + \int_\mH x\cdot\left( n(t,x) \bu(t,x) \right)\D  x \\
  & = &N M  -\iint  _{y,y'}\int _{z>0} \frac{y\cdot (y - y') + z^2}{\left(|y-y'|^2 + z^2\right)^{N/2}} n(t,y',0) n(t,y,z) \D  y \D y' \D z \, .
\end{eqnarray*}
We use a symmetrization trick to evaluate the contribution of interaction:
\begin{eqnarray*} 
& &\iint _{y,y'}\int _{z>0} \frac{y\cdot (y - y') }{\left(|y-y'|^2 + z^2\right)^{N/2}} n(t,y',0) n(t,y,z) \D  y \D y' \D z =\\ 
& &  
\frac{1}{2}\iint _{y,y'}\int _{z>0} \frac{y - y' }{\left(|y-y'|^2 + z^2\right)^{N/2}}\cdot \left(n(t,y',0) n(t,y,z) y - n(t,y,0) n(t,y',z) y'\right) \D  y \D y' \D z \, .
\end{eqnarray*} 

\begin{lemma}
Let $f$ be a smooth positive function. Assume that we have both $\partial_z f(x) \leq 0$ and
\begin{equation}  
\forall z>0\, , \, \forall y \in \R^{N-1}\, , \, \forall h\in \R^{N-1}    \quad (h\cdot y)\left( h\cdot \partial_{z}\nabla_y\log f(x)  \right)\geq 0 \, . \label{eq:cond2}
\end{equation}
Then for all $y,y' \in \R^{N-1}$ and for all $z>0$, the following inequality holds true:
\begin{equation}  
(y-y')\cdot\left(  f( y',0) f(y,z) y  - f( y,0) f( y',z)y'\right) \geq |y-y'|^2 f(y,z)f(y',z)\, . \label{eq:cond1}
\end{equation}
\end{lemma}
\begin{proof}
Inequality (\ref{eq:cond1}) rewrites as follows:
\[ (y - y')\cdot \left(   \frac{f(y,z)}{f(y ,0)}\left(1 - \frac{f(y',z)}{f(y',0)}\right) y  -  \frac{f(y',z)}{f(y',0)}\left(1 - \frac{f(y,z)}{f(y,0)}\right) y'  \right) \geq 0\, .  \]
Since $\partial_z f(x) \leq 0$ we have both $f(y ,z)\leq f(y ,0)$ and $f(y' ,z)\leq f(y' ,0)$ for all $y,y',z$. Hence we are reduced to prove that the vector field 
\[ \frac{    \frac{f(y ,z)}{f(y ,0)}}{1 - \frac{f(y ,z)}{f(y ,0)}} y \,,   \]
is monotonic with respect to the $y$ variable. Computing the derivative with respect to $y$, it is straightforward to check that it is monotonic if (\ref{eq:cond2}) is satisfied:
\begin{align*}
\nabla_y \left(\frac{    \frac{f(y ,z)}{f(y ,0)}}{1 - \frac{f(y ,z)}{f(y ,0)}} y\right)& = \left( \frac{    \frac{f(y ,z)}{f(y ,0)}}{\left(1 - \frac{f(y ,z)}{f(y ,0)}\right)^2}\right)\left( \frac{\nabla_y f(y,z)}{f(y,z)} - \frac{\nabla_y f(y,0)}{f(y,0)} \right) \otimes y \\
&\quad \quad + \left(\frac{f(y,z)}{f(y,0) - f(y,z)}\right) \mathrm{Id} \\
& \geq \left( \frac{    \frac{f(y ,z)}{f(y ,0)}}{\left(1 - \frac{f(y ,z)}{f(y ,0)}\right)^2}\right)\left( \int_{z' = 0}^z \partial_z \nabla _y\log f(y,z')\D z'  \right) \otimes y  \geq 0\, ,
\end{align*}
in the following matrix sense: $A^T + A \geq 0$.  
\end{proof}

Under the hypotheses of Theorem \ref{th2dim2} we assume that conditions (\ref{eq:cond1}) -- (\ref{eq:cond2}) are fulfilled for every time of existence.   We deduce that
\begin{eqnarray*}
 \frac{\D \bI(t) }{\D t} 
 &\leq & N M  - \frac{1}{2} \iint  _{y,y'}\int _{z>0} \frac{|y - y'|^2 + 2 z^2}{\left(|y-y'|^2 + z^2\right)^{N/2}} n(t,y',z) n(t,y,z) \D  y \D y' \D z \\
 & \leq & N M - \frac{1}{2} \iint _{y,y'}\int _{z>0} \frac{1}{\left(|y-y'|^2 + z^2\right)^{N/2-1}} n(t,y',z) n(t,y,z) \D  y \D y' \D z  \\
 \end{eqnarray*}
 Since $|y - y'|^2 + z^2 \leq 2|y|^2 + 2|y'|^2 + z^2$, and $n$ is non-negative, we have
 \begin{eqnarray*}
 \frac{\D \bI(t) }{\D t}  & \leq & N M - \frac{1}{2} \iiint   _{\left\{|y|<\frac{ R}{3}, |y'|<\frac{ R}{3}, z < \frac{2R}{3}\right\}} R^{2-N}  n(t,y',z) n(t,y,z) \D  y \D y' \D z \\ 
 & \leq & N M - \frac{R^{2-N}}{2} \int_{0<z<\frac{2R}{3}} \left( \int_{|y|<\frac{ R}{3}} n(t,y,z)\D  y \right)^2\D  z \\
 & \leq & N M - \frac{3 R^{1-N}}{4} \left( \int_{0<z<\frac{2R}{3}}   \int_{|y|<\frac{ R}{3}} n(t,y,z)\D  y\D  z \right)^2\, ,  
\end{eqnarray*}
where we have used the Cauchy-Schwarz inequality. We have therefore 
 \begin{eqnarray*}
 \frac{\D \bI(t) }{\D t}  &  \leq & N M - \frac{3 R^{1-N}}{4} \left( M -  \iint  _{\left\{z>\frac{2R}{3}\;\mbox {\footnotesize or}\;|y|>\frac{ R}{3}\right\}}     n(t,y,z)\D  y\D  z \right)^2\\
 &  \leq & N M - \frac{R^{1-N}}{2} M^2 + C R^{-N-3} \bI(t)^{2}\, ,
\end{eqnarray*}
because $R^2 < 9|x|^2$ on $\left\{z>\frac{2R}{3}\;\mbox {\footnotesize or}\;|y|>\frac R3\right\}$. 
Optimizing with respect to $R$, we conclude that the solution blows-up in finite time if $\bI(0)\leq C M^{\frac{N+1}{N-1}}$, similarly as in Section \ref{eq:sec:BU u1}.

\section{Conclusion}

Here, we have demonstrated that a class of models following \cite{HBPV} exhibit pattern formation (either blow-up or convergence towards a non homogeneous steady state) under some conditions. However we have not answered the main question: do they describe cell polarisation or not? Although the one-dimensional case is clear (spontaneous polarisation occurs if the total concentration of markers is large enough), the higher-dimensional situation is not so clear. Obviously the first model (\ref{eq:u1}) does not exhibit cell polarisation since we can integrate the equation (\ref{eq:2D model}) with respect to $z$, and we obtain for $\nu(t,y) = \int_{z>0} n(t,y,z) \D z$:
\[\partial_t \nu(t,y)  = \partial_{yy} \nu(t,y)\, .\]
Thus there is no transversal instability which is the main feature of spontaneous cell polarisation, that leads to symmetry breaking. On the other hand the  second model (\ref{eq:u2}) is expected to develop symmetry breaking as the tangential component of the advective field on the boundary is given by the Hilbert transform of the trace $n(t,y,0)$ which is known to enhance finite time aggregation at least in one dimension of space \cite{CPS}. However there is no clear mathematical distinction between the two models as continuation after the blow-up time appears to be very delicate in a similar context \cite{V1, V2,DS}. It would be very interesting to make such a distinction beyond linear analysis as performed in \cite{HBPV}. We leave it as an open question.

\medskip

\noindent{\em Acknowledgement: The authors are very grateful to M. Piel, J. Van Schaftingen and J.J.L. Vel\'azquez for stimulating discussion. VC and NM warmly thank the Centre de Recerca Matem\`atica (Barcelona) for the invitation during the special semester "Mathematical Biology: Modelling and Differential Equations" (2009).}

\bibliographystyle{plain}
\def\cprime{$'$} \def\lfhook#1{\setbox0=\hbox{#1}{\ooalign{\hidewidth
  \lower1.5ex\hbox{'}\hidewidth\crcr\unhbox0}}} \def\cprime{$'$}
  \def\cprime{$'$} \def\cprime{$'$} \def\cprime{$'$} \def\cprime{$'$}

\end{document}